\documentclass[12pt,reqno]{article}
mm \textheight=8.9in

\usepackage{amssymb}
\usepackage{amsmath}
\usepackage{young}
\usepackage{amsthm}
\usepackage{tikz-cd}
\usepackage{float}

\newcommand{\wt}{\mathrm{wt}}
\newcommand{\hgt}{\mathrm{ht}}
\newcommand{\lv}{\mathrm{lv}}
\newcommand{\id}{\mathrm{id}}

\newcommand{\Lc}{\mathcal{L}}

\newcommand{\Ac}{\mathcal{A}}

\newcommand{\Bc}{\mathcal{B}}

\newcommand{\sst}{\scriptstyle }

\newcommand{\Hg}{\mathfrak{H}}

\renewcommand{\d}{\mathfrak{d}}

\newcommand{\Ru}{\mathcal{R}}

\newcommand{\C}{\mathbb{C}}
\newcommand{\Z}{\mathbb{Z}}

\newcommand{\tp}{\otimes}

\newcommand{\V}{V}

\newcommand{\gm}{\gamma}
\newcommand{\dt}{\delta}

\newcommand{\la}{\lambda}

\newcommand{\End}{\mathrm{End}}

\newcommand{\Rm}{\mathrm{R}}

\newcommand{\btl}{\mbox{\raise1.1pt\hbox{$\scriptstyle\blacktriangleleft$}}}

\newcommand{\La}{\Lambda}

\newcommand{\g}{\mathfrak{g}}
\renewcommand{\b}{\mathfrak{b}}

\newcommand{\h}{\mathfrak{h}}

\newcommand{\eps}{\epsilon}

\newcommand{\al}{\alpha}

\newcommand{\bt}{\beta}

\newcommand{\be}{\begin{eqnarray}}
\newcommand{\ee}{\end{eqnarray}}

\newtheorem{thm}{Theorem}[section]
\newtheorem{propn}[thm]{Proposition}
\newtheorem{lemma}[thm]{Lemma}

\newtheorem{corollary}[thm]{Corollary}

\theoremstyle{definition}

\newtheorem{definition}[thm]{Definition}

\begin{document}
\title{R-matrix via Hasse diagrams}

\author{
N. Kryazhevskikh, A. Mudrov, and V. Stukopin.
 \vspace{10pt}\\
\small
 Moscow Institute of Physics and Technology,\\
\small
9 Institutskiy per., Dolgoprudny, Moscow Region,
141701, Russia,
\vspace{10pt}\\
\small
 e-mail: kriazhevskikh.nv@phystech.edu,  mudrov.ai@mipt.ru
}
\date{ }

\maketitle
\begin{abstract}
We calculate the R-matrix for  the exceptional Lie superalgebra $\d(2,1;\kappa)$
in the smallest representation of its quantum supergroup, using
a method of Hasse diagrams.
  \end{abstract}
\begin{center}
\end{center}
{\small \underline{Key words}:  R-matrix, Hasse diagrams,  exceptional Lie superalgebra}
\\
{\small \underline{AMS classification codes}: 17B10, 17B37}

\section{Introduction}
The quantum Yang-Baxter equation \cite{B} is a corner stone of quantum (super)groups discovered within the Lenningrad school \cite{FRT} and
developed into an accomplished mathematical theory by Drinfeld \cite{D1}. A solution to this equation called quasitriangular structure (universal R-matrix $\Ru$)
singles out an important class of quantized universal enveloping algebras of the underlying Lie algebras.
For a particular application, it is desirable to know its image in a finite dimensional representation.
A straightforward approach to get this image is to use an explicit factorizable formula for $\Ru$ obtained by several authors \cite{KR,KT,LS,Y1}.
On the other hand, that can be done, in some relatively simple cases, directly by an alternative method of Hasse diagrams. Such diagrams can be  associated with representations
of the Borel subalgebras constituting the triangular decomposition of the quantum group. It turns out that the method works, in principle, for every module of highest or lowest weight. It proves to be  computationally
efficient for finite dimensional modules, with the use of a computer algebra software like Maple.
In particular, it is good for the finite dimensional module that deforms the adjoint representation.

The method is utilizing the intertwining properties of the universal R-matrix in its polarized form, when
 $\Ru$ is localized in the $\hbar$-adic completion  of $U_q(\b_+)\tp U_q(\b_-)\subset U_q(\g)\tp U_q(\g)$.
Here $\b_\pm\subset \g$ are opposite Borel subalgebras in $\g$. Such a polarization is featured
in the so called standard quantization of $U(\g)$. More precisely, it is required that the R-matrix is presentable as a product $\Ru=q^{\sum_{i} h_i\tp h_i} \check{\Ru}$,
where $\{h_i\}$ is an orthonormal basis in the Cartan subalgebra $\h\subset \b_\pm$ and $\check{\Ru}\in U_q(\g_+)\tp U_q(\g_-)$.
The cutoff $\check{\Ru}$ intertwines two comultiplications in $U_q(\g)$, which identity in a one-leg representation  delivers a recursion base for the  computation algorithm.

Every module whose weights are bounded from above admits a  procedure that enables one to calculate matrix entries of $\Lc=(\pi\tp \id)(\check{\Ru})\in \End(V)\tp U_q(\b_-)$ recursively
by induction on their weight height. In general, such a recursion requires a change of basis in each step depending on a simple root (labeling an arrow of Hasse diagrams).
What is good for a module of highest weight, there is a basis which can be chosen once and for all and which simplifies the algorithm.
As a consequence, the following uniqueness theorem can be proved: for a finite dimensional $V$ of highest/lowest there exists
 a unique, up to a scalar multiple, element $L$ that features the aforementioned intertwining properties.
Sending the right tensor leg of $\Lc$ to $\End(V)$ give and subsequent multiplication by the cut off Cartan factor give the
required R-matrix. In the case of module of with lowest weight, the similar schemes applies,  with  the roles of $U_q(\g_\pm)$ obviously interchanged.

We illustrate the algorithm on the example of the quantum supergroup $U_q(\g)$ with $\g$ being the exceptional Lie superalgebra $\d(2,1;\kappa)$.
This superalgebra has diverse  applications in superconformal mechanics \cite{CHT, IKL, FIL, KL}, supergravity \cite{S},  strings \cite{BI}, and condensed matter physics
\cite{Q}.

For each of the four non-isomorphic triangular decompositions of $\g$, we calculate the R-matrix in the representation of
minimal dimension 17, which is the quantized adjoint representation of $\d(2,1;\kappa)$. Our result gives an explicit expression of $R$,
in contrast with \cite{LGT}, where $R$ is expressed through a set of five invariant projectors.

\section{Quantum supergroups}
Let $\g$ be a finite-dimensional complex  Lie superalgebra associated with a polarized root system $\Rm=\Rm^-\cup \Rm^+$ of rank $n$  with a basis $\Pi=\{\alpha_i: i \in[1,n]=I\}\subset \Rm^+$ of the simple roots and even Cartan subalgebra $\h\subset \g$. Let  $A=(a_{ij})_{1\leq i,j\leq n}$, denote the symmetrizable Cartan matrix and $(-,-)$ the corresponding
non-degenerate symmetric bilinear form on the dual vector space $\h^*$.

Fix  $q\in \C^\times$ to be not a root of unity. Define
$$[x]_q=\frac{q^x-q^{-x}}{\omega}$$
for an indeterminate $x\in \C$,
where $\omega=q-q^{-1}\not =0$ will be used  throughout the text.

For each simple root $\al\in \Pi$ define
$$
 q_{\alpha}=\begin{cases}
 q, & \mbox{if } (\alpha,\alpha)=0, \\
 q^\frac{(\alpha,\alpha)}{2}, & \mbox{if } (\alpha,\alpha)\neq0.
 \end{cases}
$$

\begin{definition}\cite{KT}
  The quantum supergroup $U_q (\g)$ is a complex unital associative superalgebra   generated by
 $ e_{\pm\alpha_i},$ and $q^{\pm h_{\alpha_i}}$ with grading
$$|e_{\pm{\alpha_i}}|=\begin{cases}
 0, & \mbox{if } i\in \varkappa\subset I  \\
 1, & \mbox{if } i\notin \varkappa.
 \end{cases},
 \quad
|q^{\pm h_{\alpha_i}}|=0,\ \forall\ i\in I,
$$ such that the following relations are satisfied:
\begin{enumerate}
  \item [(i)]	$q^{h_{\alpha_i} } q^{-h_{\alpha_i} }=q^{-h_{\alpha_i} } q^{h_{\alpha_i} }=1,$ $ q^{h_{\alpha_j} } q^{h_{\alpha_i} }=q^{h_{\alpha_i}} q^{h_{\alpha_j} }$,
\item [(ii)]	$q^{h_{\alpha_i} }  e_{\pm{\alpha_j} }  q^{-h_{\alpha_i} }=q^{\pm(\alpha_i,\alpha_j ) } e_{\pm\alpha_j }$,

\item [(iii)] $[e_{\alpha_i },e_{-\alpha_j}]=e_{\alpha_i }e_{-\alpha_j}-(-1)^{|e_{-\alpha_j }||e_{\alpha_i }|}e_{-\alpha_j }e_{\alpha_i }=\delta_{ij}[h_{\al_i}]_{q_{\al_i}}$,
\item [(iv)]  $e_{\pm \al_i}^2=0$ if $(\al_i,\al_i)=0$,
\item [(v)] $(ad_{q'} e_{\pm{\alpha_i}})^{v_{ij}} e_{\pm{\alpha_j}}=0,\quad i\neq j,\ q'=q,  q^{-1},$ where
 $$(ad_{q'} e_{\al_i}) x=e_{\al_i}x-(-1)^{|e_{\al_i}||x|}(q')^{(\alpha_i,wt(x))}x e_{\al_i},\quad x\in U_q(\g),$$
 $$ v_{ij}=\begin{cases}
                 1, & \mbox{if } (\alpha_i,\alpha_i)=(\alpha_i,\alpha_j)=0, \\
                 2, & \mbox{if } (\alpha_i,\alpha_i)=0,\ (\alpha_i,\alpha_j)\neq0,  \\
                 1-\frac{2(\alpha_i,\alpha_j)}{(\alpha_i,\alpha_i)} , & \mbox{if }(\alpha_i,\alpha_i)\neq0.
               \end{cases}$$
 \end{enumerate}
\end{definition}
\noindent
The left equality in (iii) reminds that the commutator is understood in the graded sense, while the right one imposes a relation of $U_q(\g)$.
There may be also additional Serre relations involving more than two simple root vectors \cite{Y2}. There exact form is not important for
this exposition.

In particular, (iv) implies (v) for  $j\not =i$ such that $(\al_i,\al_j)\not =0$.
If $(\al_i,\al_j)=0$, then (v) translates to $e_{\pm\al_i}e_{\pm\al_j}=-e_{\pm\al_j}e_{\pm\al_i}$.

A Hopf superalgebra structure on $U_q(\g)$ is fixed by  comultiplication defined on the generators
$e_i=e_{\alpha_i }$, $f_i=e_{-\alpha_i }$, and $q^{\pm h_i}=q^{\pm h_{\al_i}}$ as
$$
\Delta(e_i)= 1\otimes e_i+e_i\otimes q^{h_i } ,\quad \Delta(f_i)=f_i\otimes 1+ q^{-h_i }\otimes f_i,\quad \Delta(q^{\pm h_i })=q^{\pm h_i}\otimes q^{\pm h_i}.
$$
Mind that the tensor product is supercommutative: $(a\tp b)(c\tp d)=(-1)^{|b||c|}ac\tp bd$ for all homogeneous $b$ and $c$.

The counit $\eps$ is a homomorphism $U_q(\g)\to \C$ that vanishes on all $e_i$, $f_i$ and returns $1$ on $q^{\pm h_i}$.
The antipode $\gamma$ can be readily evaluated on the generators:
$$ \gamma(q^{\pm h_{i} })=q^{\mp h_{i} },\quad  \gamma(e_{i})=-e_{i}q^{-h_{i} },\quad \gamma(f_{i})=-q^{h_{i} }f_{i}.$$
It is extended as a graded anti-automorphism to the entire $U_q(\g)$.

We denote by  $U_q(\h),\ U_q(\g_+),\ \mathrm{and}\ U_q(\g_-)$ the $\Z_2$-graded $\C$-subalgebras of $U_q (\g)$ generated by $q^{\pm h_{i} },\ e_{i},$ and $f_{i}$, respectively.

\section{Hasse diagrams and quantum L-operators}
Our approach to the problem can be conveniently formulated in  the language of Hasse diagrams.
Recall that such  diagrams are associated with every partially ordered sets and vice versa. The diagrams considered   in this paper
are related with posets arising from   $U_q(\h)$-diagonalizable $U_q(\b_-)$-modules. For other applications of Hasse diagrams, see
\cite{MS1,MS2}.

The algebra $U_q(\g_-)$ is endowed with a $\Z$-grading by setting $\deg(f_\al)=1$, $\al \in \Pi$.
Let $V$ be a $U_q(\g)$-module of highest weight.   We construct the Hasse diagram $\Hg(V)=\Hg_-(V)$ as follows.
Pick up a weight basis $\{v_i\}_{i\in I}$ in $V$; this will be the set of nodes. We will also identify the nodes with elements of the index set $I$.
Denote by $\pi$ the representation homomorphism $U_q(\g)\to \End(V)$. Arrows in $\Hg(V)$ are simple roots $\al\in \Pi$; we set $j\stackrel{\al}{\longleftarrow} i$ if $\pi_{ji}(f_\al)\not =0$. It follows
that the weight  $\nu_i$ and $\nu_j$ of the nodes satisfy the equality $\nu_i-\nu_j=\al$.
We write $j\prec i$ if the nodes $i$ and $j$ can be connected by a sequence of arrows from $i$ to $j$ and call such sequence a path.
Any ordered sequence of nodes $j=i_0\prec i_1\prec\ldots \prec i_k=i$ is called a route from $i$ to $j$.


Thus arrows of the  Hasse diagram $\Hg(V)$ are coloured by simple roots. It may happen that more then one $\al$-arrows  terminate at a node $v$.
Of course, it is possible only if $\dim V[\nu+\al]>1$, where $\nu$ is the weight of $v$.
Then we say that $v$ has in-$\al$ branching. Similarly, if the number of $\al$-arrows originating at $v$ is greater than 1, we say that $v$ has out-$\al$ branching,
 see the figures below.
That may occur  only if $\dim V[\nu-\al]>1$.
\begin{figure}[h]
\begin{center}
\begin{picture}(140,50)
\put(0,20){\circle*{3}}
\put(44,-2){{\vector(-2,1){40}}}
\put(45,20){{\vector(-1,0){40}}}
\put(44,42){{\vector(-2,-1){40}}}

\put(-2,23){$v$}

\put(20,23){$\al$}\put(20,37){$\al$}\put(20,-1){$\al$}

\put(140,20){\circle*{3}}
\put(136,22){{\vector(-2,1){40}}}
\put(135,20){{\vector(-1,0){40}}}
\put(136,18){{\vector(-2,-1){40}}}

\put(138,23){$v$}

\put(100,23){$\al$}\put(110,37){$\al$}\put(110,-1){$\al$}
\end{picture}
\end{center}
  \caption{in-  and out-branching}
\label{Hasse}
\end{figure}
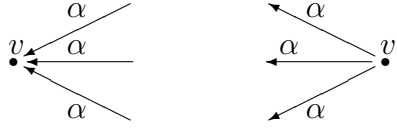
\begin{definition}
We say that an arrow $\al$ originating from a node $v$ is simple if there is no $\al$-out branching at $v$.
In other words, it is the only $\al$-arrow originating from $v$.
\end{definition}
\noindent
In a descending recursive algorithm which  we explicate below we mean by branching only the "out"-variant while ignoring the "in"-variant
as immaterial.
\begin{definition}
  We say that a module $V$ is branching free if there is a weight basis in $V$ such that its every arrow
 $\Hg(\V)$ is simple.
\end{definition}
For instance, a minuscule module (whose all weights are multiplicity free) is branching free.  Another example is  an adjoint
module for semi-simple Lie (super)algebra, see for instance Figures \ref{Hasse1} and \ref{Hasse4}.
The only possibility for branching at positive root vectors is due to the multiplicity of the zero weight equal to the rank of $\g$.
One can choose a  basis in the Cartan subalgebra, whose elements are the images of positive $\al$-root vectors via the action of negative $\al$-root vectors.
Supplemented with root vectors, they constitute a basis in $V$ without  branching.

Being branching free for a module $V$ facilitates calculation of the R-matrix, however this requirement turns out to be  too restrictive.
\begin{definition}
Suppose that $V$ has a basis $\{v_i\}_{i\in I}$ in which
every non-maximal $v_i$ is covered by   a simple arrow $v_i\stackrel{\al}{\longleftarrow} \hat v_i$.
Then $V$ is called  branching quasi free. The nodes $\hat v_i$ are called parent for $v_i$.
\end{definition}
Thus, a node may have more than one parent nodes, but a parent node has only one descendent via a given simple root.

We will assume that all weights from $\La(V)$ belong to $\la-\Gamma_+$ for some $\la\in \La(V)$, as that is the case when $V$ is of highest weight. Then we can define a function $\lv\colon \La(V)\to \Z_-$ by assigning
$\lv(\la)=0$ and $\lv(\nu)=-\hgt(\la-\nu)$. That is,  $\lv(\nu)$ is the sum of no-negative integer coordinates in the expansion of $\la-\nu$ over the basis $\Pi$.
For a node $v$ of weight $\nu$ we will write $\lv(v)=\lv(\nu)$.

\begin{propn}
 Every module of highest weight is branching quasi-free.
\end{propn}
\begin{proof}
We need to construct a basis where every node $v$ distinct from the highest has a parent node.
Denote by $d_\nu$ the multiplicity of the weight $\nu\in \La(V)$, that is,  the dimension of weight space $V[\nu]$.
We do descending induction on the level of the weight. For $v$ of $\lv(v)=-1$ the highest node is parent because $\wt(v)$ is  multiplicity free.
Suppose we have constructed a required basis in all weight spaces of level $\ell$ and pick a weight $\mu\in \La(V)$ of level $\ell-1$.

Let $\nu_i\in \La(V)$, $i\in I_\ell$, be all weights of level $\ell$ and set $\bt_i=\nu_i-\mu \in \Pi$ to be simple roots
(different for different $\nu_i$).
Let   $\{v_{i,j}\}_{j\in d_{\nu_i}}\subset V[\nu_i]$ be the already constructed basis of level $\ell$, for each $i$. Consider the set of vectors $f_{\bt_i} v_{i,j}$ over all $i\in I_\ell,j\in d_{\nu_i}$
and retain among them  independent. Clearly their span gives all $V[\mu]$ because $V$ is cyclic over $U_q(\g_-)$.
This is a basis we are seeking for: if $v=f_{\bt_i}v_{i,j}\in V[\nu_i]$  is in the chosen set, then $\hat v=v_{i,j}$ is its parent node,
as the arrow $\bt_i$ applied to $\hat v$ is simple.
\end{proof}
\noindent
For a branching quasi-free basis we fix a function $i\mapsto \hat i$ assigning to each  non-maximal node $i$ its parent node  $\hat i$ with
a simple arrow $v_i\stackrel{\bt_i}{\longleftarrow} \hat v_i$.

 Let $\Ru$ denote  a quasitriangular structure or $U_q(\g)$ relative to $\Delta$ and set $\check{\Ru}=q^{-\sum_{i}h_i\tp h_i}\Ru$, where $\{h_i\}_{i=1}^n$ is an orthonormal basis in $\h$.
It is possible to choose $\Ru$ such that $\check{\Ru}\in U_q(\g_+)\tp U_q(\g_-)$ (a completed tensor product).
The fact that the tensor factors are in the nilpotent subalgebras rather than simply Borel subalgebras can be seen from
the computation algorithm discussed in this exposition.
Furthermore, $\check \Ru=1\tp 1\mod J_+\tp J_-$, where $J_\pm$ are the ideals in $U_q(\g_\pm)$ generated by $e_{\pm \al}$, $\al\in \Pi$.

Consider two comultiplications on  $U_q(\g)$ defined on generators by the assignments
$$
\Delta(f_\al)= f_\al\tp 1+q^{-h_\al}\tp f_\al,\quad \Delta(q^{\pm h_\al})=q^{\pm h_\al}\tp q^{\pm h_\al},\quad\Delta(e_\al)= e_\al\tp q^{h_\al}+1\tp e_\al,
$$
$$
\tilde \Delta(f_\al)=f_\al\tp 1 +q^{h_\al}\tp f_\al,\quad \tilde \Delta q^{\pm h_\al}=q^{\pm h_\al}\tp q^{\pm h_\al}, \quad \tilde \Delta(e_\al)=e_\al\tp q^{-h_\al}+1\tp e_\al,
$$
for all $\al \in \Pi$.
The element $\check{\Ru} $ is a Hopf (super)algebra twist that relates the two chosen comultiplications:
\be
\label{int_rel_check_R}
\check{\Ru}\Delta(u)=\tilde \Delta(u)\check{\Ru}, \quad \forall u\in U_q(\g).
\ee
Specifically we need this relation only for the negative simple root vectors:
$$
\check{\Ru}(f_\al\tp 1+q^{-h_\al}\tp f_\al)=(f_\al\tp 1 +q^{h_\al}\tp f_\al)\check{\Ru}, \quad \forall \al\in \Pi.
$$
It is convenient to rewrite it for each ordered pair $j\preceq i$ as
$$
\check{\Ru}(f_\al\tp 1)-(f_\al\tp 1)\check{\Ru}=(q^{h_\al}\tp f_\al)\check{\Ru}-\check{\Ru}(q^{-h_\al}\tp f_\al), \quad \forall \al\in \Pi.
$$
Fix a $U_q(\g)$-module $V$ with representation homomorphism $\pi\colon U_q(\g)\to \End(V)$ and consider the matrices
$$
\Lc=(\pi\tp \id)(\check{\Ru}) \in \End(V)\tp U_q(\g_-).
$$
In a chosen  weight basis of homogeneous vectors, write it as
$$
\Lc=\sum_{j\preceq i} e_{ij}\tp L_{ji}=1\tp 1+\sum_{j\prec i} e_{ij}\tp L_{ji}
$$
with $L_{ij}\in U_q(\g_-)$.
Then the intertwining relation translates to
\be
\label{L-intertwiner}
\sum_{k\prec j} L_{ki}\pi(f_\al)_{kj}(-1)^{|\al|(|k|+|i|)}=(-1)^{|\al|(|i|+|j|)}q^{(\al,\nu_i)} f_\al L_{ji}-L_{ji}q^{-(\al,\nu_j)} f_\al+\sum_{i\prec k}\pi(f_\al)_{ik}L_{jk}.
\ee

Denote by $\End_+(V)$ the linear span $e_{ij}\in \End(V)$ with $i\succ j$. It is an associative subalgebra in $\End(V)$ invariant under left and right $U_q(\h)$-action. Similarly, by $\End_-(V)$ we denote the
subalgebra  spanned by $e_{ij}$ with $i\prec j$. The nilpotent parts of the subalgebras $U_q(\b_\pm)$ are sent to $\End_\pm(V)$ by the representation homomorphism $\pi$.
\begin{definition}
  A linear operator $A\in \End(V)\tp \End(V)$ is called polarized if
$$
A-1\tp 1\in \End_+(V)\tp \End_-(V).
$$
Similarly we call a matrix $\Ac\in \End(V)\tp U_q(\g)$ polarized if $\Ac-1\tp 1\in \End_+(V)\tp U_q(\g_-)\g_-$,
and an element $\Bc\in U_q(\g)\tp U_q(\g)$ polarized if $\Bc-1\tp 1\in U_q(\g_+)\g_+\tp U_q(\g_-)\g_-$.
\end{definition}
The algorithm we suggest for computing $\Lc$ will be arranged as a proof to the following statement.
\begin{lemma}
 The element $\Lc=\sum_{j\preceq i }e_{ij}\tp L_{ji}$ is a unique polarized matrix with entries in   $U_q(\g)$  satisfying (\ref{L-intertwiner}) for all $\al\in \Pi$.
\end{lemma}
\begin{proof}
First of all observe that $L_{ji}=0$  for $\lv(j)>\lv(i)$ and $L_{ji}=\dt_{ji}$ for  $\lv(j)=\lv(i)$; that is,
the matrix $L_{ji}$ is upper triangular. In particular, $L_{ii}=1$ for $v_i$ a maximal node.
Further we compute $L_{ji}$ by descending induction on the level of $j$.

Let  $\lv(j)\preceq \lv(i)$ and suppose that we have computed all $L_{kl}$ with $\lv(j)<\lv(k)\leqslant \lv(l)$.
 Then the summation in the l.h.s. of (\ref{L-intertwiner}) degenerates to
\be
\label{recurrent formula}
L_{ji}\pi(f_\al)_{j\hat j}(-1)^{|\al|(|j|+|i|)}=(-1)^{|\al|(|i|+|\hat j|)}q^{(\al,\nu_i)} f_\al L_{\hat ji}-L_{\hat ji}q^{-(\al,\nu_{\hat j})} f_\al+\sum_{i\prec k}\pi(f_\al)_{ik}L_{\hat jk},
\ee
where $\hat j$ is the fixed parent node for $j$ and $\al=\bt_j$ is a simple arrow.
Since $\lv(j)<\lv(\hat j)$,  the right hand side has been found by the induction assumption. As $\pi(f_\al)_{j\hat j}\not =0$, this equality gives $L_{ji}$.
\end{proof}
Thus, in order to calculate $L_{ji}$ with $j\prec i$, one should pass through a chosen  parent node $\hat j$. If the module $V$ is branching free,
that can be any preceding node of 1 level higher.

Remark that this algorithm does not prove the existence of $\Lc$, because it makes use of only a part of the intertwinging relations. The existence of $\Lc$ is {\em a priori} given.
Recall that an associative algebra is called residually finite-dimensional (RFD) if the intersection of the kernels of its finite-dimensional representations (that is a Hopf ideal) is zero.
For instance, such is the non-graded reductive quantum group.
\begin{corollary}
 Suppose that $U_q(\g)$ is RFD. Then the element $\check \Ru$ is in $U_q(\g_+)\tp U_q(\g_-)$. It is a unique polarized twist intertwining comultiplications
$\Delta$ and $\tilde \Delta$ restricted to $U_q(\g_-)$.
\end{corollary}
\begin{proof}
The matrix  $(\pi\tp \id)(\check \Ru)$  is in $\End(V)\tp U_q(\b_-)$ for all representations $\pi$ on  modules $V$
of highest weight, therefore the right tensor leg of
  $\check \Ru$ is in $U_q(\g_-)$. Similarly one can prove that the left tensor leg is in $U_q(\g_+)$, by considering the matrix
  $(\id\tp \pi)(\check \Ru)$ and  employing the intertwining relation (\ref{int_rel_check_R}) for the generators of $U_q(\g_+)$.
\end{proof}
Denote by $P\in \End(V)\tp \End(V)$  the graded flip of tensor factors,  $P(v\tp w)=(-1)^{|v||w|}w\tp v$,  and set $B=(\pi\tp \pi)(q^{\sum_i h_i\tp h_i})$.
\begin{corollary}
Let $V$ be a module of highest weight. Then the image $R$ of the universal R-matrix is the only operator in  $B+\End_+(V)\tp \End_-(V)$ such that $PR$ is $U_q(\g)$-invariant.
\end{corollary}
This implies a well known fact that  $\Ru$ is a unique quasitriangular structure on $U_q(\g)$ lying in $ U_q(\b_+)\tp U_q(\b_-)$.
Indeed, if $\Ru$ is such a structure, then the  image of $\check \Ru$  in $\End_+(V)\tp \End_-(V)$ is polarized for all modules $V$ of highest weight.

\section{The quantum supergroup $U_q\bigl(\d(2,1;\kappa)\bigr)$}
In this part we apply the algorithm outlined in the previous section to compute the R-matrix
for adjoint representation of  the quantum supergroup $U_q\bigl(\d(2,1;\kappa)\bigr)$, see \cite{FSS}. By that we mean the finite dimensional irreducible module
whose highest weight is the maximal root (the quantized adjoint module of the Lie superalgebra). Recall that, for generic $\kappa$, it is the module of
minimal dimension 17. We consider four non-isomorphic triangular polarizarions of  $U_q\bigl(\d(2,1;\kappa)\bigr)$, construct the modules for each of them and then
evaluate the R-matrices.

As we mentioned, the adjoint module is branching free, therefore all arrows are simple.  A particular
realization of the algorithm depends on a choice of  paths connecting each pair  of ordered nodes.
The results of  calculation done with of Maple are given in Section \ref{Sec_R-matrix}.
\subsection{Polarizations I-III}
\label{Sec_Pol-I-III}
The Lie superalgebra $\d(2,1;\kappa)$ has four non-isomorphic triangular decomposition. Three of them have two even simple roots while all the three
simple roots of the 4-th are odd.
It is convenient to consider the first three polarizations together, because they share the same topology of the adjoint Hasse diagram.
Their Gram matrix  $B_{ij}=(\al_i,\al_j)$, $i,j=1,2,3$, of the simple roots reads
$$
G
=
\left(
\begin{array}{ccc}
 - 2b& b  &  0\\
  b & 0 & a  \\
  0 &  a & -2a
\end{array}
\right).
$$
Specializing $(a,b)=(-\kappa,-1)$, $(a,b)=\bigl(1+\kappa,-1\bigr)$, $(a,b)=\bigl(1+\kappa,-\kappa\bigr)$,
one obtains the Gram  matrices for polarization I, II, and III, respectively. In what follows, notation $t=q^a$, $s=q^b$ is used.
Also, $2_c$  means $[2]_c$, where $c=q,s,t$.

The minimal representation of $\g$ for generic $\kappa$ of dimension $17$    is supported on the adjoint module, which we denote by $V$.
Its  highest weight is $\nu_1=\al+2\bt+\gm$.
 We explicate the Hasse diagram of $V$ regarded as a $U_q(\g_-)$ module in Figure \ref{Hasse1} below. The double circled nodes  stand for  odd vectors.
The weights of the nodes are readily figured out of the diagrams. In particular, the nodes with numbers $8,9,10$ carry zero weight.
The weights of the nodes on the right are positive roots, while of the nodes on the left are negative roots.
The non-zero matrix entries of the operators $\pi(x)$, where $x$ is a simple root vector, are given in the tables below.
So long the chosen basis has no out-branching, $\pi_{j,i}(x)=c\in \C^\times $ means $\pi(x)v_{i}=cv_j$.

The Hasse diagram corresponding to the structure of the $U_q(\g_+)$-module on $V$ is obtained
by the mirror reflection relative the vertical axis passing through the three nodes of zero weight.

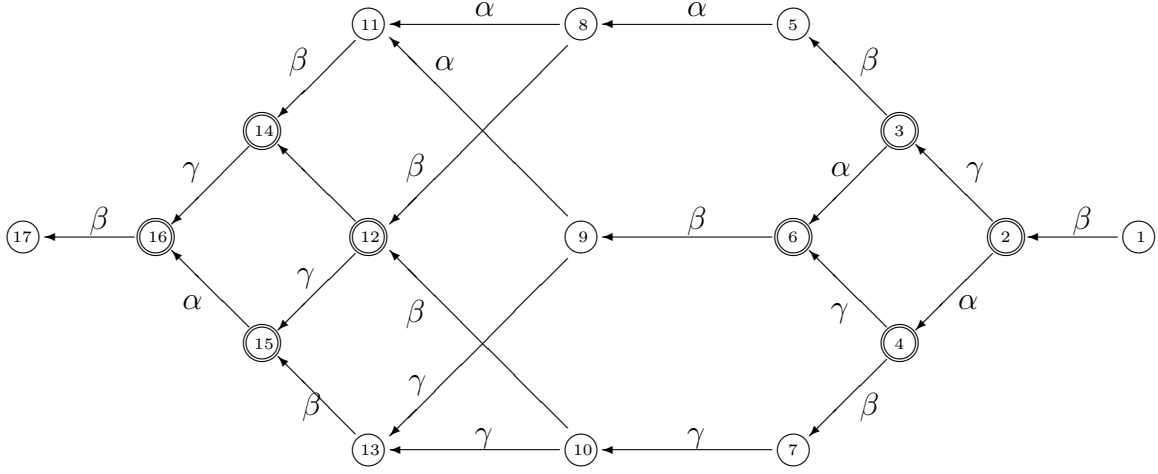
\begin{figure}[H]
\begin{center}
\begin{picture}(420,180)
\put(117,158){$\scriptscriptstyle 11$}\put(199,158){$\scriptscriptstyle 8$}\put(278,158){$\scriptscriptstyle 5$}

\put(77,118){$\scriptscriptstyle 14$}\put(318,118){$\scriptscriptstyle 3$}
\put(-14,78){$\scriptscriptstyle 17$}\put(37,78){$\scriptscriptstyle 16$}\put(117,78){$\scriptscriptstyle 12$}\put(199,78){$\scriptscriptstyle 9$}\put(278,78){$\scriptscriptstyle 6$}\put(358,78){$\scriptscriptstyle 2$}\put(409,78){$\scriptscriptstyle 1$}
\put(77,38){$\scriptscriptstyle 15$}\put(318,38){$\scriptscriptstyle 4$}

\put(117,-2){$\scriptscriptstyle 13$}\put(197,-2){$\scriptscriptstyle 10$}\put(278,-2){$\scriptscriptstyle 7$}

\put(115,154){{\vector(-1,-1){29}}}
\put(75,114){{\vector(-1,-1){29}}}\put(315,114){{\vector(-1,-1){29}}}
\put(115,74){{\vector(-1,-1){29}}}\put(355,74){{\vector(-1,-1){29}}}
\put(315,34){{\vector(-1,-1){29}}}\put(315,126){{\vector(-1,1){29}}}
\put(115,86){{\vector(-1,1){29}}}\put(355,86){{\vector(-1,1){29}}}
\put(75,46){{\vector(-1,1){29}}}\put(315,46){{\vector(-1,1){29}}}
\put(115,6){{\vector(-1,1){29}}}

\put(192,160){{\vector(-1,0){64}}}\put(272,160){{\vector(-1,0){64}}}\put(195,152){{\vector(-1,-1){67}}}
\put(32,80){{\vector(-1,0){34}}}\put(272,80){{\vector(-1,0){64}}}\put(402,80){{\vector(-1,0){34}}}\put(195,72){{\vector(-1,-1){67}}}\put(195,88){{\vector(-1,1){67}}}
\put(192,0){{\vector(-1,0){64}}}\put(272,0){{\vector(-1,0){64}}}\put(195,8){{\vector(-1,1){67}}}

\put(120,160){\circle{12}}\put(200,160){\circle{12}}\put(280,160){\circle{12}}
\put(80,120){\circle{12}}\put(80,120){\circle{13.5}} \put(320,120){\circle{12}}\put(320,120){\circle{13.5}}
\put(-10,80){\circle{12}}\put(40,80){\circle{12}}\put(40,80){\circle{13.5}}\put(120,80){\circle{12}}\put(120,80){\circle{13.5}}\put(200,80){\circle{12}}\put(280,80){\circle{12}}\put(280,80){\circle{13.5}}\put(360,80){\circle{12}}\put(360,80){\circle{13.5}}\put(410,80){\circle{12}}
\put(80,40){\circle{12}}\put(80,40){\circle{13.5}}\put(320,40){\circle{12}}\put(320,40){\circle{13.5}}
\put(120,0){\circle{12}}\put(200,0){\circle{12}}\put(280,0){\circle{12}}

\put(160,163){$\al$}\put(240,163){$\al$}
\put(90,143){$\bt$}\put(145,143){$\al$}\put(305,143){$\bt$}
\put(50,103){$\gm$}\put(134,103){$\bt$}\put(294,103){$\al$}\put(345,103){$\gm$}
\put(15,83){$\bt$}\put(385,83){$\bt$}\put(240,83){$\bt$}
\put(50,53){$\al$}\put(93,63){$\gm$}\put(134,48){$\bt$}\put(294,50){$\gm$}\put(342,53){$\al$}
\put(95,13){$\bt$}\put(305,13){$\bt$}
\put(135,23){$\gm$}
\put(160,3){$\gm$}\put(240,3){$\gm$}
\end{picture}
\end{center}
  \caption{Hasse diagram of $V$ as a $U(\g_-)$-module. Polarization I-III.}
\label{Hasse1}
\end{figure}
\begin{table}[H]
  \centering
\begin{tabular}{||c||c|c|c|c|c|c|c|c|c||}
  \hline
  $\sst(i,j)$& $\sst (4, 2)$ & $\sst(6, 3)$ & $\sst(8, 5)$ & $\sst(11, 8)$ & $\sst(11, 9)$ & $\sst(14, 12)$ & $\sst(16, 15)$&&\\ \hline
 $(f_\al)_{ij}$& 1 & 1 & 1 & 1 & $[b]_q$ & 1& 1&& \\
  \hline\hline
  $\sst(i,j)$& $\sst(2,4)$ & $\sst(3,6)$ & $\sst(5, 8)$ & $\sst(8,11)$ & $\sst(5, 9)$ & $\sst(12,14)$ & $\sst(15, 16)$&&\\ \hline
$(e_\al)_{ij}$&$-[b]_q$ &$-[b]_q$ &$-[2b]_q$ & $-[2b]_q$ & $-[2b]_q[b]_q$ & $-[b]_q$ & $-[b]_q$ &&\\\hline\hline
$\sst(i,j)$&$\sst(3,2)$ & $\sst(6, 4)$ & $\sst(10,7)$ & $\sst(13,9)$ & $\sst(13,10)$ & $\sst(15, 12)$ & $\sst(16, 14)$&&\\ \hline
$ (f_\gm)_{ij}$&  $1$ & $1$ & $1$ & $[a]_q$ & 1 & $1$ & $1$ &&\\
  \hline\hline
$\sst(i,j)$&$\sst(2,3)$ & $\sst(4,6)$ & $\sst(7,10)$ & $\sst(7,9)$ & $\sst(10,13)$ & $\sst(12, 15)$ & $\sst(14, 16)$&&\\ \hline
$(e_\gm)_{ij}$&  $-[a]_q$ & $-[a]_q$ & $-[2a]_q $& $-[2a]_q[a]_q$ & $-[2a]_q$ & $-[a]_q$ & $-[a]_q$ &&\\
  \hline\hline
 $\sst(i,j)$& $\sst(2, 1)$ & $\sst(5, 3)$ & $\sst(7, 4)$ & $\sst(9, 6)$ & $\sst(12, 8)$ & $\sst(12, 10)$ & $\sst(14, 11)$& $\sst(15, 13)$& $\sst(17, 16)$\\ \hline
$(f_\bt)_{ij}$&  $1$ & $[b]_{q^2}$ & $[a]_{q^2}$ & $\frac{1}{2_q}$ &$\frac{1}{2_s}$ & $\frac{1}{2_{t}}$ & $1$ & $1$ & $1$ \\
  \hline\hline
 $\sst(i,j)$& $\sst(1,2$ & $\sst(3,5)$ & $\sst(4,7)$ & $\sst(6,8)$ & $\sst(6,10)$ & $\sst(9,12)$ & $\sst(11,14)$& $\sst(13, 15)$& $\sst(16, 17)$\\ \hline
$ (e_\bt)_{ij}$& $[b+a]_q$& $\frac{2_q}{2_s}$ & $\frac{2_q}{2_t}$ & $\frac{2_q}{2_s}$ & $\frac{2_q}{2_t}$ &$-1$ & $-[b]_q$ & $-[a]_q$ &  $-[b+a]_q$ \\
  \hline\hline
    \end{tabular}
\caption{Adjoint representation. Polarization I-III.}
\end{table}
The generators of the  quantum supergroup $U_q(\g)$ satisfy the following commutation relations:
$$
q^{h_\bt}e_{\pm \al}q^{-h_\bt}=q^{\pm (\al,\bt)}e_{\pm \al}, \quad [e_\al,e_{-\bt}]=\dt_{\al,\bt}[h_\al]_{q}, \quad \forall \al, \bt \in \Pi.
$$
Let $\bt\in \Pi$ be odd and $\al, \gm\in \Pi$ be even, $\al\not =\gm$.
The Serre relations are
$$
e_{\pm\bt}^2=0, \quad [e_{\pm\al},e_{\pm \gm}]=0.
$$
$$
[e_{\pm\al},[e_{\pm\al},e_{\pm\bt}]_{s}]_{ s^{-1}}=0,\quad [e_{\pm\gm},[e_{\pm\gm},e_{\pm\bt}]_{t}]_{t^{-1}}=0,\quad
$$
 The matrix elements of
the generators are listed in Table 1.

\subsection{Polarization  IV}
\label{Sec_Pol-IV}
In this section we study the adjoint representation of the quantum supergroup $U_q\bigl(\d(2,1;\kappa)\bigr)$ constructed over polarization IV.
All simple roots are odd. Let us denote them by $\al=\al_1,\bt=\al_2,\gm=\al_3$.  Their Gram matrix is
$$
G=
\left(
\begin{array}{ccc}
  0& a  & 1 \\
  a & 0 & -b \\
  1 & -b & 0
\end{array}
\right).
,
$$
where $a=\kappa$, $b=1+\kappa$.
The commutation relation between the elements of the Cartan subalgebra in  $U_q(\g)$ and the Chevalley generators are
$$
q^{h_\bt}e_{\pm \al}q^{-h_\bt}=q^{\pm (\al,\bt)}e_{\pm \al}, \quad [e_\al,e_{-\bt}]=\dt_{\al,\bt}[h_\al]_{q}, \quad \forall \al, \bt \in \Pi.
$$
The simple root vectors are nilpotent,
$$
e_{\pm \al}^2=0,\quad e_{\pm \bt}^2=0,\quad e_{\pm \gm}^2=0.
$$
Positive and negative root vectors satisfy   Serre relations
$$
e_{\pm\al} e_{\pm\bt} e_{\pm\gm}-[a]_qe_{\pm\al} e_{\pm\gm} e_{\pm\bt} -[b]_q e_{\pm\gm} e_{\pm\al} e_{\pm\bt} = e_{\pm\gm} e_{\pm\bt} e_{\pm\al} -[a]_qe_{\pm\bt} e_{\pm\gm} e_{\pm\al} -[b]_q e_{\pm\bt} e_{\pm\al} e_{\pm\gm}.
$$
The Hasse diagram of the adjoint representation for $U_q(\b_-)$ of type  IV is depicted in Figure \ref{Hasse4}. The
diagram for $U_q(\b_+)$ is obtained from it by reflection  with respect to the vertical axis passing through the nodes
of zero weights.
The matrix elements of the generators are presented in Table 2.
\begin{figure}[H]
\begin{center}
\begin{picture}(420,130)

\put(137,118){$\scriptscriptstyle 11$}\put(199,118){$\scriptscriptstyle 8$}\put(258,118){$\scriptscriptstyle 5$}

\put(77,118){$\scriptscriptstyle 14$}\put(318,118){$\scriptscriptstyle 2$}
\put(16,58){$\scriptscriptstyle 17$}\put(77,58){$\scriptscriptstyle 15$}\put(137,58){$\scriptscriptstyle 12$}\put(199,58){$\scriptscriptstyle 9$}\put(258,58){$\scriptscriptstyle 6$}\put(318,58){$\scriptscriptstyle 3$}\put(379,58){$\scriptscriptstyle 1$}
\put(77,-2){$\scriptscriptstyle 16$}\put(318,-2){$\scriptscriptstyle 4$}
\put(137,-2){$\scriptscriptstyle 13$}\put(197,-2){$\scriptscriptstyle 10$}\put(258,-2){$\scriptscriptstyle 7$}

\put(80,120){\circle{12}}\put(140,120){\circle{12}}\put(140,120){\circle{13.5}} \put(320,120){\circle{12}}
\put(200,120){\circle{12}}\put(260,120){\circle{12}}\put(260,120){\circle{13.5}}
\put(20,60){\circle{12}}\put(20,60){\circle{13.5}}\put(80,60){\circle{13.5}}\put(140,60){\circle{12}}\put(140,60){\circle{13.5}}\put(200,60){\circle{12}}\put(260,60){\circle{12}}\put(260,60){\circle{13.5}}\put(320,60){\circle{12}}\put(380,60){\circle{13.5}}\put(380,60){\circle{12}}
\put(80,0){\circle{12}}\put(140,0){\circle{13.5}}\put(140,0){\circle{12}}\put(200,0){\circle{12}}\put(260,0){\circle{12}}\put(260,0){\circle{13.5}}\put(320,0){\circle{12}}

\put(374.5,66.5){\vector(-1,1){48}}
\put(372,60){\vector(-1,0){45}}
\put(374.5,53.5){\vector(-1,-1){48}}

\put(312,120){\vector(-1,0){43}}
\put(314.5,113.5){\vector(-1,-1){48}}
\put(314.5,66.5){\vector(-1,1){48}}

\put(314.5,53.5){\vector(-1,-1){48}}
\put(312,0){\vector(-1,0){43}}
\put(314.5,6.5){\vector(-1,1){48}}

\put(252,120){\vector(-1,0){43}}
\put(252,60){\vector(-1,0){43}}
\put(252,0){\vector(-1,0){43}}

\put(194.5,6.5){\vector(-1,1){48}}
\put(196,7.5){\vector(-1,2){52}}

\put(194.5,113.5){\vector(-1,-1){48}}
\put(196,112.5){\vector(-1,-2){52}}

\put(194.5,66.5){\vector(-1,1){48}}
\put(194.5,53.5){\vector(-1,-1){48}}

\put(134.5,113.5){\vector(-1,-1){48}}
\put(74.5,113.5){\vector(-1,-1){48}}
\put(134.5,66.5){\vector(-1,1){48}}
\put(134.5,53.5){\vector(-1,-1){48}}
\put(134.5,6.5){\vector(-1,1){48}}
\put(74.5,6.5){\vector(-1,1){48}}

\put(131,120){\vector(-1,0){43}}
\put(71,60){\vector(-1,0){43}}
\put(131,0){\vector(-1,0){43}}

\put(312,130){$\scriptscriptstyle \al+\bt$}
\put(312,70){$\scriptscriptstyle \al+\gm$}
\put(312,10){$\scriptscriptstyle \bt+\gm$}
\put(355,25){$\al$}\put(348,65){$\bt$}\put(355,90){$\gm$}
\put(40,25){$\al$}\put(48,65){$\bt$}\put(40,90){$\gm$}
\put(110,125){$\bt$}\put(110,5){$\bt$}
\put(290,125){$\bt$}\put(290,5){$\bt$}

\put(270,37){$\gm$}\put(270,17){$\al$}
\put(270,77){$\al$}\put(270,97){$\gm$}

\put(90,37){$\al$}\put(90,17){$\gm$}
\put(90,77){$\gm$}\put(90,97){$\al$}

\put(155,110){$\al$}\put(140,95){$\al$}

\put(155,5){$\gm$}\put(140,20){$\gm$}
\put(145,40){$\bt$}\put(145,70){$\bt$}

\put(230,125){$\al$}\put(230,65){$\bt$}\put(230,5){$\gm$}
\end{picture}
\end{center}
  \caption{Hasse diagram of $V$ as a $U(\g_-)$-module, $\b_-$ of type IV}
\label{Hasse4}
\end{figure}
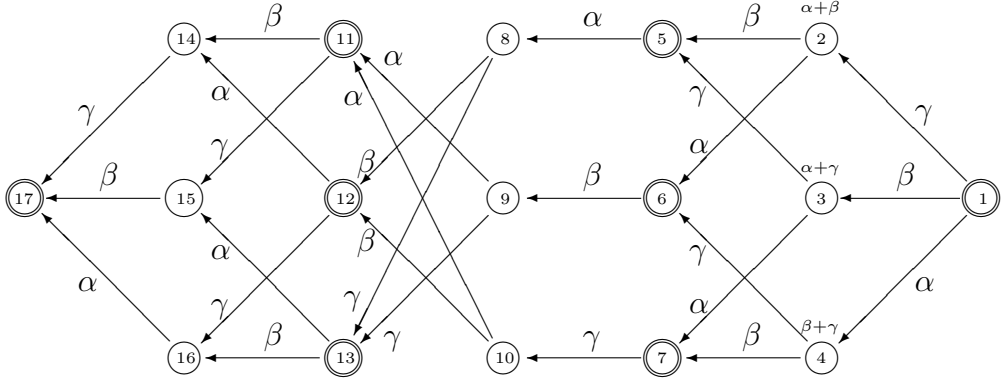
\begin{table}[H]
  \centering
\begin{tabular}{||c||c|c|c|c|c|c|c|c|c||}
  \hline
  $\sst(i,j)$& $\sst(4,1)$ & $\sst(6,2)$ & $\sst(7,3)$ & $\sst(8,5)$ & $\sst(11,9)$ & $\sst(11,10)$ & $\sst(14, 12)$&$\sst(15,13)$&$\sst(17,16)$\\ \hline
 $(f_\al)_{i,j}$& 1 & 1 & $-\frac{1}{[b]_q} $& 1 & $1$ &  $-\frac{[b]_{q}}{[a]_{q}^2}$ & $1$ &$-\frac{[b]_{q}}{[a]_{q}}$&$\frac{[b]_{q}^2}{[a]_{q}^2}$\\
  \hline\hline
  $\sst(i,j)$& $\sst(1,4)$ & $\sst(2,6)$ & $\sst(3,7)$ & $\sst(5, 9)$ & $\sst(5,10)$ & $\sst(8,11)$ & $\sst(12, 14)$&$\sst(13,15)$&$\sst(16,17)$\\ \hline
$(e_\al)_{ij}$&$[b]_q$ &$[a]_q$ &$-[b]_q$ & $-[a]_q$ & $\frac{[b]_q}{[a]_q}$ & $[a]_q$ & $-[a]_q$&$\frac{[a]_q}{[b]_q}$&$-\frac{[a]_q^2}{[b]_q}$ \\\hline\hline
$\sst(i,j)$&$\sst(3,1)$ & $\sst(5,2)$ & $\sst(7,4)$ & $\sst(9,6)$ & $\sst(12,8)$ & $\sst(12,10)$ & $\sst(14,11)$& $\sst(16,13)$& $\sst(17,15)$\\ \hline
$(f_\bt)_{ij}$&  $1$ & $1$ & $1$ & $1$ & $1$ & $\frac{[b]_q^2}{[a]_q^2}$ & $1$ &$1$&$\frac{1}{[a]_q}$\\
  \hline\hline
$\sst(i,j)$&$\sst(1,3)$ & $\sst(2,5)$ & $\sst(4,7)$ & $\sst(6,8)$ & $\sst(6,10)$ & $\sst(9,12)$ & $\sst(11,14)$&$\sst(13,16)$&$\sst(15,17)$\\ \hline
$(e_\bt)_{ij}$&  $-1$ & $[a]_q$ & $-[b]_q$ & $-[a]_q$ & $-\frac{[b]_q^2}{[a]_q}$ & $[a]_q$ & $-[a]_q$ &$[b]_q$&$[a]_q$\\
  \hline\hline
 $\sst(i,j)$& $\sst(2, 1)$ & $\sst(5, 3)$ & $\sst(6, 4)$ & $\sst(10,7)$ & $\sst(13, 8)$ & $\sst(13, 9)$ & $\sst(15, 11)$& $\sst(16, 12)$& $\sst(17, 14)$\\ \hline
$(f_\gm)_{ij}$&  $1$ & $\frac{1}{[a]_q}$ & $-\frac{[b]_q}{[a]_q}$ & $1$ &$-\frac{1}{[b]_q}$ & $1$ & $1$ & $-\frac{[a]_q}{[b]_q}$ & $1$ \\
  \hline\hline
 $\sst(i,j)$& $\sst(1,2)$ & $\sst(3,5)$ & $\sst(4,6)$ & $\sst(7,8)$ & $\sst(7,9)$ & $\sst(10,13)$ & $\sst(11,15)$& $\sst(12, 16)$& $\sst(14, 17)$\\ \hline
$(e_\gm)_{ij}$&  $-[a]_q$ & $[a]_q$ & $[a]_q$ & $\frac{[a]_q}{[b]_q}$ &$-[a]_q$ & $[a]_q$ & $-1$ & $-\frac{[b]_q^2}{[a]_q}$ & $[a]_q$ \\
  \hline\hline
\end{tabular}
\caption{Adjoint representation. Porarization IV.}
\end{table}
\section{R-matrix for $U_q\bigl(\d(2,1;\kappa)\bigr)$}
\label{Sec_R-matrix}
In this section, we present the R-matrices   in the adjoint representation,  for all the four quantum groups $U_q\bigl(\d(2,1;\kappa)\bigr)$.
We express it as the product
$$
R=HL,
$$
where
$
H=(\pi\tp \pi)(q^{\sum_{l,k=1}^3G_{lk}^{-1} h_l\tp h_k})=\sum_{i,j}^{17}q^{(\nu_i,\nu_j)}e_{ii}\tp e_{jj}
$
and
 $L=\sum_{i,j=1}^{17}e_{ij}\tp \ell_{ji}$ is the truncated operator $(\id \tp \pi)(\Lc)$. The entries $\ell_{ji}=\pi(L_{ji})\in \End(V)$ are given in the subsequent sections.

The relations of the parameters $a$ and $b$ to $\kappa$ is  as specified  the Sections \ref{Sec_Pol-I-III} and \ref{Sec_Pol-IV}.
Besides, we set  $t=q^a$, $s=q^b$, and introduce  a shortcut
$$
\lceil x  \rceil = x-x^{-1},
$$
in order to compactify the formulas.
We also use the notation
$$
[x,y]_c =xy-c yx, \quad \{x,y\}_c =xy+c yx, \quad c\in \C,
$$
for any entities $x,y$.
\subsection{Truncated L-operator in polarizations I-III}
This section contains formulas for matrix entries of the $L$-operator in polarizations I-III.
We denote them by $\ell_{i,j}$ regarding as matrices from $\End_-(V)$. In order to shorten
the formulas, we express each $\ell_{i,j}$  recursively through matrix entries of higher level.
Their coincidence  with the entries  delivered by the algorithm was
also checked by Maple.
$$\ell_{2,1}=-\lceil st \rceil f_\bt;\quad
\ell_{3,2}=-\lceil t\rceil f_\gm;\quad
\ell_{4,2}=-\lceil s\rceil f_\al;\quad
\ell_{5,3}=-\frac{\lceil q^2\rceil}{2_s}f_\bt;$$
$$\ell_{6,3}=-\lceil s\rceil f_\al;\quad
\ell_{6,4}=-\lceil t\rceil f_\gm;\quad
\ell_{7,4}=-\frac{\lceil q^2\rceil}{2_t}f_\bt;\quad
\ell_{8,5}=-\lceil s^2\rceil f_\al;$$
$$\ell_{9,5}=-\frac{\lceil s^2\rceil\lceil s\rceil}{\lceil q\rceil}f_\al;\quad
\ell_{8,6}=-\frac{\lceil q^2\rceil}{2_s}f_\bt;\quad
\ell_{10,6}=-\frac{\lceil q^2\rceil}{2_t}f_\bt;\quad
\ell_{9,7}=-\frac{\lceil t^2\rceil\lceil t\rceil}{\lceil q\rceil}f_\gm;$$
$$\ell_{10,7}=-\lceil t^2\rceil f_\gm;\quad
\ell_{11,8}=-\lceil s^2\rceil f_\al;\quad
\ell_{12,9}=\lceil q\rceil f_\bt;\quad
\ell_{13,10}=-\lceil t^2\rceil f_\gm;$$
$$\ell_{14,11}=\lceil s\rceil f_\bt;\quad
\ell_{14,12}=-\lceil s\rceil f_\al;\quad
\ell_{15,12}=-\lceil t\rceil f_\gm;\quad
\ell_{15,13}=\lceil t\rceil f_\bt;$$
$$\ell_{16,14}=-\lceil t\rceil f_\gm;\quad
\ell_{16,15}=-\lceil s\rceil f_\al;\quad
\ell_{17,16}=\lceil st\rceil f_\bt;
$$
$$\ell_{3,1}=[f_\gm,\ell_{2,1}]_t;\quad
\ell_{4,1}=[f_\al,\ell_{2,1}]_s;\quad
\ell_{5,2}=\frac{\lceil q^2\rceil}{2_s\lceil t\rceil}[f_\bt,\ell_{3,2}]_t;\quad
\ell_{6,2}=-\lceil s\rceil\ell_{3,2}f_\al;$$
$$\ell_{7,2}=\frac{\lceil q^2\rceil}{2_t\lceil s\rceil}[f_\bt,\ell_{4,2}]_s;\quad
\ell_{8,3}=s^{-1}[f_\al,\ell_{5,3}]_{s^3};\quad
\ell_{9,3}=-2_q s[f_\bt,\ell_{6,3}]_{s^{-1}};\quad
\ell_{10,3}=\frac{\lceil q^2\rceil}{2_t\lceil s\rceil}[f_\bt,\ell_{6,3}]_s;$$
$$\ell_{8,4}=\frac{\lceil q^2\rceil}{2_s\lceil t\rceil}[f_\bt,\ell_{6,4}]_t;\quad
\ell_{9,4}=-2_q t[f_\bt,\ell_{6,4}]_{t^{-1}};\quad
\ell_{10,4}=t^{-1}[f_\gm,\ell_{7,4}]_{t^3};\quad
\ell_{11,5}=-\lceil s\rceil s^{-1}\ell_{8,5}f_\al;$$
$$\ell_{12,5}=-[f_\bt,\ell_{8,5}]_s;\quad
\ell_{11,6}=2_s[f_\al,\ell_{8,6}]_s;\quad
\ell_{13,6}=2_t[f_\gm,\ell_{10,6}]_t;\quad
\ell_{12,7}=-[f_\bt,\ell_{10,7}]_t;$$
$$\ell_{13,7}=-\frac{\lceil q\rceil}{t}f_\gm\ell_{9,7};\quad
\ell_{14,8}=-[f_\bt,\ell_{11,8}]_s;\quad
\ell_{14,9}=[f_\al,\ell_{12,9}]_s;\quad
\ell_{15,9}=[f_\gm,\ell_{12,9}]_t;$$
$$\ell_{15,10}=-[f_\bt,\ell_{13,10}]_t;\quad
\ell_{16,11}=[f_\gm,\ell_{14,11}]_t;\quad
\ell_{16,12}=-\lceil s\rceil\ell_{15,12}f_\al;\quad
\ell_{16,13}=[f_\al,\ell_{15,13}]_s;$$
$$\ell_{17,14}=-\frac{\lceil st\rceil}{\lceil t\rceil}[f_\bt,\ell_{16,14}]_t;\quad
\ell_{17,15}=-\frac{\lceil st\rceil}{\lceil s\rceil}[f_\bt,\ell_{16,15}]_s;$$

$$\ell_{5,1}=-t\frac{\lceil q^2\rceil}{2_s}f_\bt\ell_{3,1};\quad
\ell_{6,1}=[f_\al,\ell_{3,1}]_s;\quad
\ell_{7,1}=-s\frac{\lceil q^2\rceil}{2_t}f_\bt\ell_{4,1};$$
$$\ell_{8,2}=s^{-1}[f_\al,\ell_{5,2}]_{s^3};\quad
\ell_{9,2}=2_q\bigl([\ell_{6,2},f_\bt]_{st}-\ell_{6,1}\bigr);\quad
\ell_{10,2}=t^{-1}[f_\gm,\ell_{6,2}]_{t^3};$$
$$\ell_{11,3}=s^{-1}[f_\al,\ell_{8,3}]_s;\quad
\ell_{12,3}=s\lceil q\rceil f_\bt\ell_{9,3};\quad
\ell_{13,3}=2_t[f_\gm,\ell_{10,3}]_t;\quad
\ell_{11,4}=2_s[f_\al,\ell_{8,4}]_s;$$
$$\ell_{12,4}=\lceil t^2\rceil t^2 f_\bt\ell_{10,4};\quad
\ell_{13,4}=t^{-1}[f_\gm,\ell_{10,4}]_t;\quad
\ell_{14,5}=-\lceil s\rceil\ell_{12,5}f_\al;\quad
\ell_{15,5}=[f_\gm,\ell_{12,5}]_t;$$
$$\ell_{14,6}=s\lceil s\rceil f_\bt\ell_{11,6};\quad
\ell_{15,6}=t\lceil t\rceil f_\bt\ell_{13,6};\quad
\ell_{14,7}=[f_\al,\ell_{12,7}]_s;\quad
\ell_{15,7}=-\lceil t\rceil\ell_{12,7}f_\gm;$$
$$\ell_{16,8}=[f_\gm,\ell_{14,8}]_t;\quad
\ell_{16,9}=[f_\al,\ell_{15,9}]_s;\quad
\ell_{16,10}=[f_\al,\ell_{15,10}]_s;\quad
\ell_{17,11}=t\lceil st\rceil f_\bt\ell_{16,11};$$
$$\ell_{17,12}=-s[f_\al,\ell_{17,14}]_{s^{-1}};\quad
\ell_{17,13}=s\lceil st\rceil f_\bt\ell_{16,13};
$$
$$\ell_{8,1}=[f_\al,\ell_{5,1}]_{s^2};\quad
\ell_{9,1}=-2_q st\{f_\bt,\ell_{6,1}\}_{(st)^{-1}};\quad
\ell_{10,1}=[f_\gm,\ell_{7,1}]_{t^2};\quad
\ell_{11,2}=s^{-1}[f_\al,\ell_{8,2}]_s;$$
$$\ell_{12,2}=-\frac{\lceil q\rceil}{\lceil st\rceil}\{f_\bt,\ell_{9,2}\}_{st};\quad
\ell_{13,2}=t^{-1}[f_\gm,\ell_{10,2}]_t;\quad
\ell_{14,3}=-\lceil s\rceil\ell_{12,3}f_\al;$$
$$\ell_{15,3}=-s\{f_\bt,\ell_{13,3}\}_{\frac{t}{s}};\quad
\ell_{14,4}=-t\{f_\bt,\ell_{11,4}\}_{\frac{s}{t}};\quad
\ell_{15,4}=-\lceil t\rceil f_\gm\ell_{12,4};\quad
\ell_{16,5}=-\lceil s\rceil\ell_{15,5}f_\al;$$
$$\ell_{16,6}=s[f_\al,\ell_{15,6}]+\ell_{15,3};\quad
\ell_{16,7}=[f_\al,\ell_{15,7}]_s;\quad
\ell_{17,8}=-\{f_\bt,\ell_{16,8}\}_{st};$$
$$\ell_{17,9}=\frac{\ell_{16,6}}{2_q}-\{f_\bt,\ell_{16,9}\}_{st};\quad
\ell_{17,10}=-\{f_\bt,\ell_{16,10}\}_{st};$$
$$\ell_{11,1}=[f_\al,\ell_{8,1}];\quad
\ell_{12,1}=-2_s st[f_\bt,\ell_{8,1}]_{(st)^{-1}};\quad
\ell_{13,1}=[f_\gm,\ell_{10,1}];\quad
\ell_{14,2}=s^{-1}[f_\al,\ell_{12,2}]_{s^2};$$
$$\ell_{15,2}=t^{-1}[f_\gm,\ell_{12,2}]_{t^2};\quad
\ell_{16,3}=s^{-1}[f_\al,\ell_{15,3}]_{s^2};\quad
\ell_{16,4}=t^{-1}[f_\gm,\ell_{14,4}]_{t^2};$$
$$\ell_{17,5}=[\ell_{17,8},f_\al];\quad
\ell_{17,6}=st\lceil st\rceil f_\bt\ell_{16,6};\quad
\ell_{17,7}=[\ell_{17,10},f_\gm];
$$
$$\ell_{14,1}=[f_\al,\ell_{12,1}]_s;\quad
\ell_{15,1}=[f_\gm,\ell_{12,1}]_t;\quad
\ell_{16,2}=t^{-1}[f_\gm,\ell_{14,2}]_{t^2};$$
$$\ell_{17,3}=-s[f_\bt,\ell_{16,3}]_t;\quad
\ell_{17,4}=-t[f_\bt,\ell_{16,4}]_s;
$$
$$\ell_{16,1}=[f_\al,\ell_{15,1}]_s;\quad
\ell_{17,2}=[\ell_{17,3},f_\gm]_t;$$
$$\ell_{17,1}=-ts\{f_\bt,\ell_{16,1}\}.$$

\subsection{Truncated L-operator in polarization IV}
This subsection is using the same notational convention as the previous. It gives
entries of the $L$-operator of in polarization IV.
$$\ell_{2,1}=\lceil t\rceil f_\gm;\quad
\ell_{3,1}=\lceil q\rceil f_\bt;\quad
\ell_{4,1}=-\lceil qt\rceil f_\al;\quad
\ell_{5,2}=-\lceil t\rceil f_\bt;$$
$$\ell_{6,2}=-\lceil t\rceil f_\al;\quad
\ell_{5,3}=-\lceil t\rceil f_\gm;\quad
\ell_{7,3}=\lceil qt\rceil f_\al;\quad
\ell_{6,4}=-\lceil t\rceil f_\gm;$$
$$\ell_{7,4}=\lceil qt\rceil f_\bt;\quad
\ell_{9,5}=\lceil t\rceil f_\al;\quad
\ell_{10,5}=-\frac{\lceil qt\rceil\lceil q\rceil}{\lceil t\rceil}f_\al;\quad
\ell_{8,6}=\lceil t\rceil f_\bt;$$
$$\ell_{10,6}=\frac{\lceil qt\rceil^2}{\lceil t\rceil}f_\bt;\quad
\ell_{8,7}=-\frac{\lceil t\rceil\lceil q\rceil}{\lceil qt\rceil}f_\gm;\quad
\ell_{9,7}=\lceil t\rceil f_\gm;\quad
\ell_{11,8}=-\lceil t\rceil f_\al;$$
$$\ell_{12,9}=-\lceil t\rceil f_\bt;\quad
\ell_{13,10}=-\lceil t\rceil f_\gm;\quad
\ell_{14,11}=\lceil t\rceil f_\bt;\quad
\ell_{15,11}=\lceil q\rceil f_\gm;$$
$$\ell_{14,12}=\lceil t\rceil f_\al;\quad
\ell_{16,12}=\frac{\lceil qt\rceil^2}{\lceil t\rceil}f_\gm;\quad
\ell_{15,13}=-\frac{\lceil t\rceil\lceil q\rceil}{\lceil qt\rceil}f_\al;\quad
\ell_{16,13}=-\lceil qt\rceil f_\bt;$$
$$\ell_{17,14}=-\lceil t\rceil f_\gm;\quad
\ell_{17,15}=-\lceil t\rceil f_\bt;\quad
\ell_{17,16}=\frac{\lceil t\rceil^2}{\lceil qt\rceil}f_\al;$$
$$\ell_{5,1}=-q^{-1}\{f_\bt,\ell_{2,1}\}_{\frac{q}{t}};\quad
\ell_{6,1}=-qt\{f_\al,\ell_{2,1}\}_{\frac{1}{qt^2}};\quad
\ell_{7,1}=-q^{-1}\{f_\bt,\ell_{4,1}\}_{q^2t};$$
$$\ell_{8,2}=-\{f_\bt,\ell_{6,2}\}_t;\quad
\ell_{9,2}=-\{f_\al,\ell_{5,2}\}_t;\quad
\ell_{10,2}=-\frac{\lceil qt\rceil}{q\lceil t\rceil}\{f_\bt,\ell_{6,2}\}_{q^2t};$$
$$\ell_{8,3}=-q\{f_\al,\ell_{5,3}\}_{q^{-1}};\quad
\ell_{9,3}=qt\{f_\al,\ell_{5,3}\}_{\frac{1}{qt^2}};\quad
\ell_{10,3}=-q\{f_\gm,\ell_{7,3}\}_{q^{-1}};$$
$$\ell_{8,4}=q^{-1}\{f_\bt,\ell_{6,4}\}_{\frac{q}{t}};\quad
\ell_{9,4}=-(qt)^{-1}\{f_\bt,\ell_{6,4}\}_{qt};\quad
\ell_{10,4}=-(qt)^{-1}\{f_\gm,\ell_{7,4}\}_{qt};$$
$$\ell_{12,5}=\{f_\bt,\ell_{9,5}\}_t;\quad
\ell_{13,5}=-\frac{\lceil qt\rceil}{\lceil t\rceil}\{f_\gm,\ell_{9,5}\}_q;\quad
\ell_{11,6}=\{f_\al,\ell_{8,6}\}_t;\quad
\ell_{13,6}=-\frac{\lceil qt\rceil}{\lceil t\rceil}\{f_\gm,\ell_{8,6}\}_{(qt)^{-1}};$$
$$\ell_{11,7}=\frac{\lceil t\rceil}{\lceil q\rceil}\{f_\al,\ell_{8,7}\}_q;\quad
\ell_{12,7}=\frac{\lceil t\rceil}{\lceil q\rceil}\{f_\bt,\ell_{8,7}\}_{(qt)^{-1}};\quad
\ell_{14,8}=-\{f_\bt,\ell_{11,8}\}_t;\quad
\ell_{15,8}=-\{f_\gm,\ell_{11,8}\}_q;$$
$$\ell_{14,9}=-\{f_\al,\ell_{12,9}\}_t;\quad
\ell_{16,9}=\frac{\lceil qt\rceil}{\lceil t\rceil}\{f_\gm,\ell_{12,9}\}_{(qt)^{-1}};$$
$$
\ell_{15,10}=\frac{\lceil t\rceil}{\lceil qt\rceil}\{f_\al,\ell_{13,10}\}_q;\quad
\ell_{16,10}=-\{f_\bt,\ell_{13,10}\}_{(qt)^{-1}};$$
$$\ell_{17,11}=-q^{-1}\{f_\gm,\ell_{14,11}\}_{\frac{q}{t}};\quad
\ell_{17,12}=-qt\{f_\gm,\ell_{14,12}\}_{\frac{1}{qt^2}};\quad
\ell_{17,13}=-\frac{\lceil t\rceil\,qt}{\lceil q\rceil}\{f_\bt,\ell_{15,13}\}_{\frac{1}{q^2t}};$$
$$\ell_{8,1}=-qt[f_\al,\ell_{5,1}]_{\frac{1}{qt}};\quad
\ell_{9,1}=-q^{-1}[f_\bt,\ell_{6,1}]_q;\quad
\ell_{10,1}=-t^{-1}[f_\gm,\ell_{7,1}]_t;$$
$$\ell_{11,2}=-t\lceil t\rceil f_\al\ell_{8,2};\quad
\ell_{12,2}=-\lceil t\rceil f_\bt\ell_{8,2};\quad
\ell_{13,2}=\frac{\lceil t\rceil}{q\lceil q\rceil}[f_\bt,\ell_{13,5}]_q;$$
$$\ell_{11,3}=-\lceil t\rceil\ell_{8,3}f_\al;\quad
\ell_{12,3}=t^{-1}[f_\gm,\ell_{12,5}]_t;\quad
\ell_{13,3}=-q\lceil t\rceil f_\gm\ell_{10,3};$$
$$\ell_{11,4}=-[\ell_{11,6},f_\gm]_{t^{-1}};\quad
\ell_{12,4}=-\lceil t\rceil\ell_{9,4}f_\bt;\quad
\ell_{13,4}=\lceil qt\rceil f_\gm\ell_{9,4};$$
$$\ell_{14,5}=t\lceil t\rceil f_\al\ell_{12,5};\quad
\ell_{15,5}=-\frac{\lceil t\rceil\lceil q\rceil}{\lceil qt\rceil}\ell_{13,5}f_\al;\quad
\ell_{16,5}=-\frac{\lceil qt\rceil}{\lceil q\rceil}[f_\bt,\ell_{13,5}]_{q^{-1}};$$
$$\ell_{14,6}=t\lceil t\rceil f_\bt\ell_{11,6};\quad
\ell_{15,6}=\frac{\lceil q\rceil}{\lceil t\rceil}[f_\gm,\ell_{11,6}]_{t^{-1}};\quad
\ell_{16,6}=-(qt)^{-1}\lceil qt\rceil f_\bt\ell_{13,6};$$
$$\ell_{14,7}=\frac{\lceil t\rceil}{\lceil q\rceil}[f_\bt,\ell_{11,7}]_{q^{-1}};\quad
\ell_{15,7}=q\lceil q\rceil f_\gm\ell_{11,7};\quad
\ell_{16,7}=\frac{\lceil qt\rceil^2}{qt\lceil t\rceil}f_\gm\ell_{12,7};$$
$$\ell_{17,8}=-[f_\gm,\ell_{14,8}]_{t^{-1}};\quad
\ell_{17,9}=-[f_\gm,\ell_{14,9}]_{t^{-1}};\quad
\ell_{17,10}=-\frac{\lceil t\rceil}{\lceil q\rceil}[f_\bt,\ell_{15,10}]_{q^{-1}};$$
$$\ell_{11,1}=-qt\{f_\al,\ell_{9,1}\}_{\frac{1}{qt}};\quad
\ell_{12,1}=-q^{-1}\{f_\bt,\ell_{8,1}\}_q;\quad
\ell_{13,1}=-t^{-1}\{f_\gm,\ell_{9,1}\}_t;$$
$$\ell_{14,2}=-2t f_\al\ell_{12,2};\quad
\ell_{15,2}=\frac{\lceil t\rceil\lceil q\rceil}{\lceil qt\rceil}\ell_{13,2}f_\al;\quad
\ell_{16,2}=-q^{-1}\lceil qt\rceil f_\bt\ell_{13,2};$$
$$\ell_{14,3}=-\lceil t\rceil\ell_{12,3}f_\al;\quad
\ell_{15,3}=-2q\ell_{11,3}f_\gm;\quad
\ell_{16,3}=\frac{\lceil qt\rceil^2}{t\lceil t\rceil}f_\gm\ell_{12,3};$$
$$\ell_{14,4}=-\lceil t\rceil\ell_{11,4}f_\bt;\quad
\ell_{15,4}=-\lceil q\rceil\ell_{11,4}f_\gm;\quad
\ell_{16,4}=-\frac{2}{qt}\ell_{13,4}f_\bt;$$
$$\ell_{17,5}=\{f_\al,\ell_{17,8}\}_{qt};\quad
\ell_{17,6}=-\frac{\lceil t\rceil}{\lceil q\rceil}\{f_\bt,\ell_{15,6}\}_{q^{-1}};\quad
\ell_{17,7}=-\{f_\gm,\ell_{14,7}\}_{t^{-1}};$$
$$\ell_{14,1}=-qt[f_\al,\ell_{12,1}]_{q^{-1}};\quad
\ell_{15,1}=-t^{-1}[f_\gm,\ell_{11,1}]_{qt};\quad
\ell_{16,1}=-q^{-1}[f_\bt,\ell_{13,1}]_{t^{-1}};$$
$$\ell_{17,2}=-\frac{\lceil t\rceil\,t}{\lceil q\rceil}[f_\bt,\ell_{15,2}]_{\frac{1}{qt}};\quad
\ell_{17,3}=-q[f_\gm,\ell_{14,3}]_{\frac{1}{qt}};\quad
\ell_{17,4}=-(qt)^{-1}[f_\gm,\ell_{14,4}]_q;$$
$$\ell_{17,1}=-\frac{\lceil t\rceil}{q\lceil q\rceil}\{f_\bt,\ell_{15,1}\}.$$

\vspace{20pt}

\noindent
\underline{\large \bf Acknowledgement}

\vspace{10pt}
\noindent
This work is done at the Center of Pure Mathematics MIPT. It is financially supported by Russian Science Foundation grant 26-11-00115.

 \end{document}